\documentclass[final,3p]{elsarticle}

\usepackage{lineno}

\usepackage[colorlinks]{hyperref}
\usepackage{graphicx,amsmath,amssymb,amsthm,fixmath,color}

\usepackage{epstopdf}

\journal{Ecological Complexity}

\newcommand{\R}{{\Bbb R}}

\DeclareMathOperator{\Prob}{\mathrm{Prob}}

\newtheorem{theorem}{Theorem}[section]
\newtheorem{lemma}[theorem]{Lemma}
\newtheorem{proposition}[theorem]{Proposition}
\newtheorem{corollary}[theorem]{Corollary}
\newtheorem{assumption}[theorem]{Assumption}
\newtheorem{principle*}{Decrease Biodiversity Principle}

\theoremstyle{definition}

\newtheorem{example}{Example}

\theoremstyle{remark}
\newtheorem{remark}[theorem]{Remark}

\begin{document}

\begin{frontmatter}

\title{Biodiversity and robustness of large ecosystems
}


\author[kozlov]{Vladimir Kozlov}
\ead{vladimir.kozlov@liu.se}

\author[vakula]{Serge Vakulenko}
\ead{vakulenfr@gmail.com}

\author[uno]{Uno Wennergren}
\ead{Uno.Wennergren@liu.se}


\author[kozlov]{Vladimir Tkachev\corref{mycorrespondingauthor}}
\cortext[mycorrespondingauthor]{Corresponding author}
\ead{vladimir.tkatjev@liu.se}

\address[kozlov]{Department of Mathematics, Link\"oping University, Sweden}
\address[vakula]{Inst. for Mech. Eng. Problems, Russian Acad. Sci.,  St. Petersburg,   ITMO University, St. Petersburg and Saint Petersburg State University of Industrial Technology and Design }
\address[uno]{Department of Physics, Chemistry, and Biology, Link\"oping University, Sweden}

\begin{abstract}
We study the  biodiversity  problem for resource
competition systems with extinctions and self-limitation effects.
Our main result establishes  estimates of biodiversity in terms of the  fundamental parameters of the model. We also prove the global stability of solutions for systems with extinctions and large turnover rate.  We show that when the extinction threshold is distinct from zero, the large time dynamics of system  is fundamentally non-predictable. In the last part of the paper  we obtain explicit analytical estimates of ecosystem robustness with respect to variations of resource supply which support the $R^*$ rule for a system with random parameters.
\end{abstract}


\end{frontmatter}



\section{Introduction} \label{intro}

Existence and stability of large foodwebs, where many species share
a few of resources,   is one of key problems in ecology
\cite{Hardin, Volterra, Hu61} as well as extinctions and mass
extinctions in such systems under  climate changes
\cite{Rothman}.  In this paper, we consider the model initiated in
\cite{KVU16}, \cite{KTVW17} describing   an ecological system, where
several (many) species  compete or fight for few limited resources.

The most typical examples are plant or plankton ecosystems.
Sunlight, water, nitrogen, phosphorus and iron are all abiotic
essential resources for phytoplankton and plant species.
Resource competition models  link the population
dynamics of competing species with the dynamics
of the resources. As it was mentioned in \cite{HuWe99}  an
attractive feature of resource competition models is that
they use the biological traits of species to predict the time
evolution of competition. In fact, many rigorous results
\cite{Hsu05,  Tilman1980, SmithWaltman}
show that, in general situation, a single species survives and to
obtain coexistence of many species one needs very special
assumptions to species parameters (mortalities and resource
consumption rates).  This paradox (the so-called paradox of plankton
\cite{Hu61, Hardin}) has received a great attention in past decades
\cite{Volm, Roy2007}. Several ways to explain
 the extreme diversity of phytoplankton communities  have
been proposed. In particular, the proposed mechanisms  include
spatial and temporal heterogeneity in physical and biological
environments,  horizontal turbulence of ocean,  oscillation and
chaos generated by several internal and external causes, stable
coexistence and compensatory dynamics under fluctuating temperature
in resource competition, and  toxin-producing
phytoplankton \cite{Volm, Roy2007}. Although the mathematical problem
has been studied for more than two decades it
is still far from to be well-understood. The most of available
results do not give   explicit estimates of biodiversity in terms of
the fundamental observable ecosystem parameters (such as species
mortality rates, rates of resource consumptions, resource supply and
resource turnover rate).

The main goal of this paper is to present such estimates.  To this
end we consider dynamical equations are close to the model
equations, which considered in
the seminal paper \cite{HuWe99}  but  extend that model  in the two
aspects. First, \  we take into account
 self-limitation effects  (which are important for plankton
 populations \cite{Roy2007} and to explain stability of large foodwebs
\cite{Alles2, Alles1}). Roughly speaking when we introduce a weak
self-limitation we replace equations of Maltus type on Verhulst type
equations. Second, following \cite{KVU16} we take into account
species extinction thresholds, however, in contrast to
\cite{KVU16,VakSud} we consider here the case of a few resources.
Mathematically, our approach with extinction thresholds and
self-limitation terms can be considered as a regularization of
resource competition models.

Our main results can now be formulated as follows. A summary of
the mathematical framework and the global stability results established earlier in \cite{KVU17a}, \cite{KTVW17} for the model with zero extinction
threshold  is collected in Sections~\ref{model} and \ref{Gbstab}. In
that case, a complete  description of the system large time
behaviour is obtained for systems with sufficiently large turnover rates and without extinctions. More precisely, the model exhibits the \textit{global stability}: all positive trajectories converge to the same equilibrium state. This result holds due to two principal properties of our system. First, the system has a typical fast/slow structure for large turnover rates. Second, the system obeys a monotonicity property: if resources increase then species abundances also increase. We recall the principal ideas of the proof at the end of section~ \ref{equ}.

Next, if one allows
even small positive extinction threshold, the ecosystem  behaviour
exhibits  new interesting  effects. We study this  in sections~
\ref{ext} and \ref{bioest} below. We establish a weaker
stability result: the limit equilibrium state still exists but it
depends on the initial ecosystem  state. This in particular implies that there can a priori exist several distinct equilibrium states.

In section~\ref{bioest}, we establish explicit upper and below estimates of biodiversity expressed in terms of the fundamental ecosystem parameters  (such as species mortalities, resource consuming rate etc.). Remarkably, the obtained estimates are universal for small extinction thresholds and
self-limitation parameters. We point out that these results use no assumptions on the system dynamics and do not use our theorem on global stability.

In the  part of the paper, we study large ecosystems with random fundamental parameters. Here the main assumption  is that the system dynamics has no oscillating or chaotic regimes. Note that it follows from Theorem~\ref{Theor1}, that the assumption is automatically holds if, for example, turnovers rates are large enough. Recall that the $R^*$ rule (also called the resource-ratio hypothesis) is a hypothesis in community ecology that attempts to predict which species will become dominant as the result of competition for resources. It predicts that if multiple species are competing for a single limiting resource, then species, which survive at the lowest equilibrium resource level,  outcompete all other species \cite{Til2}, \cite{Fisher}. In section~\ref{Rstar} we  obtain a complete description of parameters for  survived species and establish the validity of  the $R^*$ rule for systems with random parameters. We show that if the  resources are limited and initially the number of  species is sufficiently large then only species with  the fitness which is  close to maximal one  can survive. In our model, the  fitness is determined as the resource amounts available for an organism.

Finally, in section \ref{extmass}, we study sensitivity of those
states with respect to a change of environmental parameters.
This allows us to essentially extend  recent results of
\cite{Rothman}. Namely, not only the magnitude of environmental
changes and their rates determine how much species will extinct but
also the achieved biodiversity level, and some species parameters.
For example, ecosystems where the species parameters are localized
at some values are less stable than ecosystems with a large species
parameter variation.

\section*{The basic notation}

\begin{tabbing}
\hspace{55mm}\= \=\hspace*{70mm} \\
$x(t)=(x_1(t),\ldots,x_M(t))$\> the vector of species abundances \\
$v(t)=(v_1(t),\ldots,v_m(t))$\> the vector of  resource abundances \\
$\mu_i$, $\gamma_{i} $\>  the mortality and the self-limitation  constant of species $i$\\
$D_j$, $S_j$\> the turnover rate and the supply of resource $v_j$\\
$c_{ij}$\> the content  of resource $j$ in species $i$\\
$\phi_i$\> the specific growth rates  of species $i$\\
$K_{ij}$\> the half-saturation constant for resource $j$ of species $i$, page~\pageref{Kref}\\
$(x^{\mathrm{eq}},v^{\mathrm{eq}})$\> the special equilibrium state, page~\pageref{xvref}\\
$X_{\mathrm{ext}}^{(i)}$\> the extinction threshold of  species $i$, page~\pageref{Xex}\\
$N_e(t)$\> the number of species which exist at the time $t$, page~\pageref{Nref}
\end{tabbing}


\section{Preliminaries}  \label{model}

Given $x,y\in \mathbb{R}^n$ we  use the standard vector order
relation: $x\le y$ if $x_i\le y_i$ for all $1\le i\le n$,
$x< y$  if $x\le y$ and $x\ne y$, and $x\ll y$ if $x_i< y_i$ for all
$i$; $\R_{+}^n$ denotes the nonnegative cone $ \{x\in\R^n:x\ge 0\}$
and for $a\le b$, $a,b\in\R^n$
$$
[a,b]=\{x\in \R^{n}:a\le x\le b\}
$$
is the closed box with vertices at $a$ and $b$.

We consider the following system of equations:
\begin{align}
     \frac{dx_i}{dt}&=x_i (\phi_i(v)- \mu_i  -  \gamma_{i} \; x_i),
     \quad\quad\quad\quad\quad\! i=1,\dots, M,
    \label{HX1}\\
     \frac{dv_k}{dt}&=D_k(S_k -v_k)   -  \sum_{i=1}^M c_{ ki} \; x_i
     \; \phi_i(v), \quad k=1,\dots, m.
    \label{HV1}
     \end{align}
Here $x=(x_1, x_2,\ldots , x_M)$ is the vector of  species
abundances and $v=(v_1,\ldots , v_m)$ is a vector of
resource amounts, where $v_k$ is the resource of  $k$-th type
consumed by all ecosystem species, $\mu_i$ are the species
mortalities, $D_k >0$ are resource turnover rates, $S_k$ is the
supply of the resource $v_k$,  and
$c_{ik} >0$ is the content  of $k$-th resource in the $i$-th
species.
The coefficients $\gamma_{i} >0$ describe self-limitation effects
\cite{Roy2007}, \cite{KVU16}, \cite{KTVW17}.

We  consider general $\phi_j$  which are bounded, non-negative and
Lipshitz continuous
\begin{equation}
        |\phi_j(v) -\phi_j(\tilde v)| \le L_j \|v -\tilde v\|
\label{MM2a}
     \end{equation}
and
\begin{equation}
      \phi_k(v) =0,   \quad  \text{for all  $ k$ and $v \in \partial
      \mathbb{R}^m_{+}$}.
\label{MM2b}
     \end{equation}
 We use the norm notation $\|x\|=\max_{1\le i\le m}|x_i|$.

Furthermore, we shall assume   that each  $\phi_k(v)$ is a
non-decreasing function of each variable $v_j$ in $\mathbb{R}^M_{+}$.
This assumption means that as the amount of $j$-th resource
increases all the functions $\phi_l$ also increase.

Conditions \eqref{MM2b}  and (\ref{MM2a}) can be interpreted as a
generalization
of the well known von Liebig law, where
\begin{equation} \label{Liebm}
      \phi_i(v) =r_i\min \Big \{  \frac{ v_1}{K_{i1} +  v_1},
      \ldots,  \frac{ v_m}{K_{im} +  v_m}\Big  \}
     \end{equation}
where $r_{i}$ and $K_{ij}$  are positive coefficients, $i=1,\ldots ,
M$. Here, $r_i$ is the maximal level of the resource consumption rate by  $i$-th  species and $K_{ij}$ is the half-saturation constant for resource $j$ of species $i$.\label{Kref}

The Liebig law  can be considered as a generalization of  Holling functional response (Michaelis-Menten kinetics) for the case of many resources.  It assumes that the species growth is determined  by the scarcest resource (limiting factor). In particular, the Liebig law can  be applied to ecosystem models for resources such as  sunlight or mineral nutrients, for example, for plant ecosystems. For the case of a single resource $m=1$ and $v=v_1 \in \mathbb{R} $ it reduces to the Holling response. In this case, a typical example of $\phi_i$ satisfying all above conditions
is
\begin{equation} \label{Liebs}
      \phi_i(v) =  \frac{r_{i} v}{K_{i} +  v}, \quad  i=1, \ldots,
      M.
     \end{equation}

For $\gamma_{i}=0$   system  \eqref{HX1}, \eqref{HV1} was considered
in the studies of the  plankton paradox, see, for example,
\cite{HuWe99}.  Following \cite{Roy2007} and \cite{KTVW17} we assume
$\gamma_i >0$ since it is known that self-limitation is essential
for large ecosystem \cite{Alles1, Alles2} and
plankton or plant ecosystems can induce effects leading to
self-limitation.
We complement system \eqref{HX1},  \eqref{HV1} by non-negative
initial conditions
\begin{equation} \label{Idata}
x(0)=\bar x, \quad   v(0)=\bar v,
\end{equation}
where
\begin{equation}\label{K16a}
\bar x_i>0,\; i=1,\ldots,M,\;\;\mbox{and}\;\;0 \leq \bar v_k  \le
S_k, \quad  k=1, \ldots , m.
\end{equation}

\section{Estimates and equilibria  }  \label{Gbstab}

In this section, we study the stability and large time behavior of
solutions to the Cauchy problem \eqref{HX1}, \eqref{HV1} and \eqref{Idata}. Using the standard partial order relations in $\mathbb{R}^m$, we write  $v \le w$ if  $  v_i \le w_i$ for each $i$, and $v\ll w$ if $  v_i < w_i$ for each $i$. We also denote $z_{+}=\max\{z, 0\}$.

\subsection{Boundedness of solutions}
 Let  $S=(S_1,\ldots,S_m)$. The proof of the following estimates can be found in  \cite{KTVW17}.

\begin{proposition} \label{prop1}
Solution $(x, v)$ of \eqref{HX1}, \eqref{HV1} with  initial data
\eqref{K16a}
 is well defined for all positive $t$,  and it satisfies the
 estimates
\begin{equation}
 0\le x_i(t) \le \frac  {  \bar x_i \exp(\bar a_i t) }{1 +  \bar x_i
 \gamma_i \bar a_i^{-1}(\exp(\bar a_i t)  -1)}, \quad i=1,\ldots ,
 M,
\label{estx}
     \end{equation}
where $\bar a_i =\phi_i(S) - \mu_i$, and
\begin{equation}
 0\le v_k(t) \le    S_k (1- \exp(-D_kt)) +  \bar v_k  \exp(-D_kt),
 \quad k=1,\ldots , m.
\label{estv}
     \end{equation}
Furthermore, if $\phi_k(S)\leq\mu_k$ for some $k$ then
$\lim_{t\to\infty}x_k(t)=0$.
\end{proposition}

In what follows, we make the following  natural \textbf{assumption}:
\begin{equation}\label{K1}
\phi_i(S)-\mu_i>0,\;\;  \text{for all $i=1,\ldots , M$}.
\end{equation}
Indeed,  if $\phi_i(S)-\mu_i \le 0$ for a certain index $i$ then
according to Proposition~\ref{prop1} $x_i(t) \to 0 $  as
$t\to\infty$, thus,  species $i$ will surely go extinct and
therefore  can be excluded from the analysis.

\subsection{Equilibrium resource values and convergence to
equilibria }\label{equ}
Let $E$ denote the set of nonnegative equilibrium points (stationary
solutions) $(x,v)$ of \eqref{HX1}-\eqref{HV1}. It is straightforward to see that $(0,S)\in E$. This point expresses the equilibrium resource
availabilities in the absence of any species. Furthermore, it was
shown in \cite{KTVW17} that under assumption \eqref{K1}, for an
arbitrary $(x,v)\in E$ such that $(x,v)\ne (0,S)$ there holds
\begin{equation}\label{positivity}
x>0 \quad \text{and}\quad 0\ll v\ll S.
\end{equation}
Among all equilibrium points in $E$, we shall distinguish the
\textit{special} ones defined as follows: an equilibrium  point
$(x^{\mathrm{eq}},v^{\mathrm{eq}})$ is called \textit{special} if  $v^{\mathrm{eq}}$ is a solution\label{xvref}
of the fixed point problem
\begin{equation}
  D_k(S_k -v^{\mathrm{eq}}_k)=F_{k}(v^{\mathrm{eq}}),\qquad 1\le
  k\le m,
\label{Starvd}
     \end{equation}
where
$$
F_{k}(v)=\sum_{i=1}^{M} \frac{c_{ki}
\phi_i(v)}{\gamma_i}(\phi_i(v) -\mu_i)_{+},
$$
and $x^{\mathrm{eq}}$ is uniquely determined by
\begin{equation} \label{sur66}
x^{\mathrm{eq}}_i=\frac1{\gamma_i} (\phi_i(v^{\mathrm{eq}})
-\mu_i)_{+}.
\end{equation}
Note that $F_k$ can be interpreted as total consuming rates $\frac{1}{\gamma_i}c_{ki}
\phi_i(v)(\phi_i(v) -\mu_i)_{+}$ over
all the species $\le i\le M$. As a corollary of the monotonicity of $\phi_i$, it can be shown that the set of special equilibrium points $(x^{\mathrm{eq}},\,v^{\mathrm{eq}})$ is
nonempty for any choice of the fundamental parameters of the model, see  \cite{KTVW17} for the proof.
If $m=1$ then one can easily prove that there always exists a unique
special point, see \cite{KVU16}. But, if $m\ge 2$ and $D_i$ or
$\gamma_i$ are small enough, there can exist several special points.
On the other hand, under some natural conditions, there exists a
unique special equilibrium point. More precisely, we have the
following result.

 \begin{theorem}[Corollary 7.2 in \cite{KTVW17}] \label{Theor1}
Assume that
\begin{equation}\label{Lip}
\rho:=\max_{1 \le i \le m} \sum_{j=1}^M  \frac{c_{ij} L_j
(2\phi_j(S)-\mu_i)}{D_i \gamma_j} \leq 1
\end{equation}
Then all the solutions $(x(t), v(t))$ of \eqref{HX1}, \eqref{HV1}
with positive initial data
$v(0) > 0, x(0) > 0$ converge, as $t \to \infty$, to   the unique
special equilibrium point defined by
\eqref{Starvd} and \eqref{sur66}.
\end{theorem}

In Section  \ref{equ}  we will see that the unique point
$(x^{\mathrm{eq}},\,v^{\mathrm{eq}})$  is the right equilibrium
point for the system (\ref{HX1})-(\ref{HV1}), which attracts all
trajectories. The proof of Theorem~\ref{Theor1} is given in our recent paper \cite{KTVW17}. Our approach is based on an iteration technique and two side estimates for period-two-points of nonincreasing maps and relies on the monotone properties of our system. More precisely, eliminating in an appropriate way the $x$-variables, we reduce the analysis to certain integral equations for  $v=(v_1,\ldots,v_m)$. It turns out that these integral equations  give rise to a nonincreasing operator in an appropriate Banach space  of functions $0\le v\le S$. By the monotonicity, the solution $v$ is estimated by the iterated sequences $\bar v^{(2n)}$ and   $\bar v^{(2n+1)}$, which are shown to converge, say to limits $v_{even}$ and $v_{odd}$, respectively. The second part of the proof establishes the estimates on the difference $v_{even}-v_{odd}$ in terms of $\rho$ in \eqref{Lip} above. This implies that if the turnover rates $D_i$ are large enough then  the limits $v_{even}$ and $v_{odd}$ coincide, which yields the required stability result.

Alternatively, the main idea of the proof can be explained as follows. If the turnover rate $D=\min_i \{D_i\}$ are large enough then our system has the standard structure typical for systems with slow/fast variables. More precisely, in that case the resources $v_i$ evolve fast in time whereas the species abundances are slow variables. Using this property we make substitution $v_i=S_i- \tilde v_i$, where  $\tilde v_i$ are new unknowns. This yields an a priori estimate $|\tilde v_i|=O(D^{-1})$. Therefore, given $x_i(t)$ we can solve the differential equations for  $v_i$ by iterations. The a priori estimates also imply that the first approximation for $x_i$ can be determined by certain simplified equations which do not involve $v$. Proceeding further, we get  an approximation for $v_i$ of order $O(D^{-2})$, etc. This yields the convergence of the procedure by standard methods of theory of invariant  manifolds for slow/fast systems.

\section{Model with extinction threshold} \label{ext}

\subsection{Formulation of the model}
\label{Xex}
System (\ref{HX1})-(\ref{HV1})  does not take into account  species
extinctions due to extinction thresholds. To describe this effect,
we follow \cite{KVU16} with certain simplifications. In order to describe species extinction  we introduce small positive
parameters $X_{\mathrm{ext}}^{(i)}$ being an extinction threshold
and we will consider only solutions such that
$x_i(0) > X_{\mathrm{ext}}^{(i)}$. Estimates of $X_{\mathrm{ext}}^{(i)}$ can be obtained by stochastic
models, for example, \cite{Nasell}, however, these models are rather complicated.  The approach suggested below simplifies the problem and still allows us to obtain nontrivial effects.

Let $ S_e(t)$ denote the set of indices of those species which exist
at the time $t$ and let $N_e(t)$ denote the cardinality of $S_e(t)$.
Then $S_e(0)=\{1,2,\ldots , M\}$ and $N_e(0)=M$.

We say that $x_k$ disappears at the moment $t_*$ if
$$
x_k(t_*)=X_{\mathrm{ext}}^{(k)}\quad \text{  and }
\quad x_k(t) > X_{\mathrm{ext}}^{(k)}\quad \text{ for  $t < t_*$.}
$$
In this case we remove the index $k$ from the set $ S_e(t)$, $t\geq
t_*$. The parameters $X_{\mathrm{ext}}^{(k)}$ can be interpreted as
a thresholds for species abundances.


With modifications described above,  equations (\ref{HX1}) and (\ref{HV1})
define the dynamics as follows. Within each time interval $(t_*,
T_*)$ between the subsequent species extinctions,  the dynamical
evolution of $x_i(t)$ is determined by the system
(\ref{HX1}), (\ref{HV1}). We obviously have
$$
S_e(t_2)\subset S_e(t_1), \quad  0\le t_1 \le t_2,
$$
therefore $N_e(t)$ is a piecewise constant decreasing function,
hence
there exist  limits
\begin{equation} \label{limit}
S_e(t) \to S_e, \quad N(t) \to N_{e},  \quad \text{as }t \to
+\infty,
\end{equation}
where $N_e$ is the cardinality of the set $S_e$.\label{Nref}  Note that the set
$S_e$ and  its cardinality  $N_{e}$ depend on the initial data
as it is shown in \cite{KVU16}. By \eqref{limit}  there exists a
time moment $T_e$  such that \textit{all extinctions have occurred}
and thus we can use Theorem~\ref{Theor1} for the remained species.
For $t > T_e$ system \eqref{HX1}, \eqref{HV1} can be rewritten as
follows:
\begin{align}
     \frac{dx_i}{dt}&=x_i ( \phi_i(v)- \mu_i  -  \gamma_{i} \; x_i),
     \quad  i \in S_{e},
    \label{HX1E}\\
     \frac{dv_k}{dt}&=D_k(S_k -v_k)   -  \sum_{i \in S_{e}} c_{ ki}
     \; x_i \; \phi_i(v), \quad k=1,\dots, m,
    \label{HV1E}
     \end{align}
where
\begin{equation} \label{star}
x_i(t)  > X_{\mathrm{ext}}^{(i)}, \quad \forall t > T_e.
\end{equation}
This system describes large time dynamics of ecosystem when all
extinctions are finished. Note that the set of remaining species
$S_{e}$ may be different for different initial data of original
system (\ref{HX1}), (\ref{HV1}).

First let us consider the particular case of
\eqref{HX1E}-\eqref{HV1E} is when all species disappear, i.e.
$S_e=\emptyset$. Then \eqref{HX1E}-\eqref{HV1E} amounts to
$$
\frac{dv_k}{dt}=D_k(S_k -v_k)\quad k=1,\dots, m,
$$
hence the solution converges to the supply equilibrium state
$(x,v)=(0,S)$. The dynamic in this case is trivial, therefore, we
assume in what follows that
$$
S_e\ne \emptyset.
$$

An analysis of the dependence of $S_e$ and $N_e$ on the initial data
seems to be rather complicated, therefore, we consider instead the
maximal possible number $N_{\mathrm{max}}$ of species which may
survive. More precisely, we define
 \begin{equation} \label{Nmax}
N_{\mathrm{max}} =\max  N_e
\end{equation}
where the maximum is taken over all possible sets $S_e\subset
\{1,\ldots,M\}$ for which the problem (\ref{HX1E})-(\ref{star}) has
a solution. We present some estimates of $N_{\mathrm{max}}$ in
section~\ref{bioest} below.

\subsection{Dynamics of the model with extinction thresholds}
\label{AsymExt}

To obtain equilibria for system (\ref{HX1E}), (\ref{HV1E}),  we
should take into account that $x_i(t) > X_{\mathrm{ext}}^{(i)}$ for
all $t > T_f$. Therefore, equations for equilbria take
the form
\begin{equation} \label{sur66e}
X_i(v^{\mathrm{eq}})=\gamma_i^{-1}
\left(\phi_i(v^{\mathrm{eq}}) -\mu_i\right)_{+,\epsilon_i}, \quad
\epsilon_i=\gamma_i X_{\mathrm{ext}}^{(i)}.
\end{equation}
and
\begin{equation}
  D(S -v^{\mathrm{eq}})=F_{\mathrm{ext}}(v^{\mathrm{eq}}, b, K, p),
\label{StarvdM}
     \end{equation}
where
\begin{equation}  \label{FNeq}
F_{\mathrm{ext}}(v, b, K, p)=\sum_{i\in S_{e}} R_i (v, b, K, p).
\end{equation}
Here and in what follows, we denote by $z_{+, \delta}$ the cut-off
function
$$
z_{+, \delta}=
\left\{
\begin{array}{rl}
0& \text{if $z < \delta$};\\
z& \text{if $z \ge \delta$};\\
\end{array}
\right.
$$

Since the dynamics after the moment $T_f$  is completely determined
by differential equations  (\ref{HX1E}), (\ref{HV1E}), we obtain  by
Theorem~\ref{Theor1}:

 \begin{corollary} \label{colII}
Suppose that $S_e$ is not empty and condition \eqref{Lip} holds. Let
$(x, v)$ be solution  of \eqref{HX1E}, \eqref{HV1E} and
\eqref{star}.
Then this solution converges, as $t \to \infty$, to   an
equilibrium point satisfied   \eqref{sur66e}, \eqref{StarvdM}, and
\eqref{FNeq}.
\end{corollary}

\section{Estimates of biodiversity} \label{bioest}

In the first three subsections of this section we derive
general estimates of biodiversity which make no assumptions on large time behaviour of the system. In section~\ref{2ex}, we obtain the biodiversity estimates which are asymptotically sharp. In the  remaining part, we also establish estimates under assumption that there are not oscillating or chaotic regimes, and each trajectory converge to an equilibria.

\subsection{The general case}

In the model with extinctions we are able to estimate the maximal
biodiversity, i.e.  the maximal possible value $N_e$ expressed by
means of the ecosystem parameters. For this purpose we apply the
averaging procedure similar to that considered in \cite{KVU16} for
$m=1$.

Let $f(t)$ be  a continuous function which is uniformly bounded on
$[0, +\infty)$. Then its $t$-average value $\langle f \rangle$ is
defined by
$$
\langle f \rangle_t = \frac{1}{t} \int_0^{t} f(s) ds.
$$
 Let $x=(x_1,\ldots,x_{N_e})$, $v=(v_1,\ldots,m)$ be a solution of
 (\ref{HX1E})-(\ref{star}). Applying  the $t$-average to
 (\ref{HX1E}) and (\ref{HV1E}) and using the boundedness of $x$ and
 $v$, we obtain
 \begin{equation}\label{LUD1}
 \langle x_i\phi_i(v) \rangle_t-\mu_i\langle x_i \rangle_t-\gamma_i
 \langle x_i^2 \rangle_t=\xi_i(t)
 \end{equation}
 and
\begin{equation}  \label{averv}
   D_k(S_k - \langle v_k \rangle_t ) = \sum_{i \in S_{e}} c_{ki}
   \langle x_i  \phi_i(v) \rangle_t + \tilde \xi_k(t),
\end{equation}
where
$\xi_i(t)=\frac{1}{t}(x_i(t)-x_i(0))$, $\tilde
\xi_k(t)=\frac{1}{t}(v_k(t)-v_k(0))$.
Since $x_i$ and $v_k$ are nonnegative and bounded from above,  we obtain $\lim_{t \to +\infty}\xi_i(t)=\lim_{t \to +\infty}\tilde
\xi_k(t)=0$.

Next, since $x_i(t) $ is bounded from above and $x_i(t) > X_{\mathrm{ext}}^{(i)}>0$, we have
$$
\eta_i(t):=\left\langle \frac{x_i'}{x_i} \right\rangle_t = \frac{1}{t} \int_0^{t} \frac{dx_i(s)}{x_i(s)}=\frac{1}{t}\ln \frac{x_i(t)}{x_i(0)}\to 0\quad \text{as}\quad t\to+\infty,
$$
therefore dividing the left-hand and right-hand sides of \eqref{HX1E} by
$x_i(t)$ followed by the $t$-average we get
\begin{equation}  \label{averx}
   \langle \phi_i(v) \rangle_t  = \mu_i + \gamma_i \langle x_i
   \rangle_t + \eta_i(t),
\end{equation}
where $\eta_i(t) \to 0$ as $t \to +\infty$. Furthermore, by the Cauchy
inequality we have
$
\langle x_i \rangle_t^2\leq \langle x_i^2 \rangle_t,
$
hence we derive from (\ref{LUD1}) that
\begin{equation}  \label{phiv}
\langle x_i  \phi_i \rangle_t \ge (\mu_i + \gamma_i
X_{\mathrm{ext}}^{(i)}) \langle x_i   \rangle_t + \xi_i(t).
\end{equation}

Relations \eqref{averx} and \eqref{phiv}  allow us to obtain a general estimate of consumed resources $v_k$ expressed only in the fundamental parameters of the main system (not involving $\gamma_i$ and $ X_{\mathrm{ext}}^{(i)}$).

\begin{lemma} \label{lemma1}
If the von Liebig law \eqref{Liebm}
holds, we have for sufficiently large $t$ and  all  $\le k\le m$
\begin{equation}  \label{vvv1}
\langle   v_k \rangle_t  > V_k(S_e) := \max_{i \in S_e} \frac{\mu_i
K_{ik} }{r_i - \mu_i}.
 \end{equation}
\end{lemma}

\begin{proof} Let $T>0$ be chosen such that $|\eta_i(t)|<\gamma_i
X_{\mathrm{ext}}^{(i)}$ for all $t>T$. Next, note that the function $\phi_i(v)$ defined by \eqref{Liebm} is concave as  the minimum of concave functions.  Therefore, combining  Jensen's inequality  with \eqref{averx} we obtain
$$
\phi_i(\langle v \rangle_t )\ge \langle \phi_i(v) \rangle_t > \mu_i,\quad \forall t>T
$$
which implies the desired estimate \eqref{vvv1}.
\end{proof}

\begin{remark}
Note that in the case  $m=1$  one has
\begin{equation}  \label{vvv2}
\langle   v \rangle_t  >  \bar V= \min_{i \in \{1,\ldots , M\}}
\lambda_i,
 \end{equation}
where
\begin{equation}  \label{lambd}
 \lambda_i =\frac{ \mu_i K_i}{r_i - \mu_i}.
\end{equation}
The parameters  $\lambda_i$ represent break-even concentrations and
appear in analysis of resource competition ecosystems as important
species characteristics (see  \cite{Hsu05} and the references
therein).
\end{remark}

Now, combining  (\ref{vvv1}) with (\ref{phiv}) it follows  from
(\ref{averv}) that for large $t$
 \begin{equation}\label{LUD2}
 D_k(S_k - V_k)  \geq \sum_{i \in S_{e}} c_{ki} (\mu_i + \gamma_i
 X_{\mathrm{ext}}^{(i)}) \langle x_i   \rangle_t+ \xi_k(t),
\end{equation}
where $\xi_k(t) \to 0$ as $t \to +\infty$. Let us introduce the
following averages:
$$
\theta_k(S_e):=\frac{1}{N_e}\limsup_{t\to\infty}\sum_{i \in S_{e}}
c_{ki} (\mu_i + \gamma_i X_{\mathrm{ext}}^{(i)}) \langle x_i
\rangle_t,
$$
and also define
$$
\bar \theta_k=\min_{S_e \ne \emptyset}  \theta_k(S_e), \quad \bar
V_k=\min_{S_e\ne \emptyset}  V_k(S_e),
$$
where we take the minima over all possible sets $S_e \ne \emptyset$.
Then (\ref{LUD2}) implies

\begin{proposition} \label{prop2} {The number of survived species
satisfies
\begin{equation*}  \label{averv2}
      {N_{e}} \le \min_{k}\frac{D_k (S_k-  \bar V_k)} { \bar
      \theta_{k}}.
\end{equation*}
}
\end{proposition}

Note that the latter estimate involves only the main observable
fundamental ecosystem parameters. Furthermore, the obtained estimate
is universal in the sense that it is valid without any assumptions
on $D_k$, the number of resources  and the ecosystem dynamics.

\subsection{Modifications of the main estimate}

Estimate \eqref{averv} can be simplified for small $\gamma_i
X_{\mathrm{ext}}^{(i)} \ll \mu_i$. In that case, one has
\begin{equation*}  \label{averv6}
      {N_{e}} \le  \min _{k}   \frac{D_k (S_k- \bar V_k)}{
      \hat{\theta}_k },
\end{equation*}
where
$$
\hat{\theta}_k=\min_{S_e \ne \emptyset}
\frac{1}{N_e}\limsup_{t\to\infty}\sum_{i \in S_{e}} c_{ki} \mu_i
\langle x_i   \rangle_t.
$$

Taking into account that $x_i(t) > X_{\mathrm{ext}}^{(i)}$, we also
have a  rough estimate
\begin{equation}  \label{averv61}
      {N_{e}} \le \min_{k}   \frac{D_k (S_k- \bar V_k)}{  Z_k},
\end{equation}
where
$$
Z_k= \min_{S_e \ne \emptyset} {N_e}^{-1}\sum_{i \in S_{e}}  c_{ik}
\mu_i  X_{\mathrm{ext}}^{(i)}.
$$

\subsection{Estimates of biodiversity from below}
Now we want to estimate the biodiversity from below.  We consider
the case of a single resource $m=1$. To this end, let us introduce
the set
$$
{\mathcal B}_M=\Big \{i \in \{1,\ldots , M\}:    \frac{r_i \bar V}{K_{i}
+S}  > \mu_i \ \text{for all}
\ i \Big \},
$$
where $\bar V$ is defined by \eqref{vvv2}. The cardinality of a set
$A$ will be denoted by $|A|$.

\begin{proposition} \label{prop3}
\

\begin{itemize}
\item[{\bf a})]
Consider the dynamic without extinctions defined by \eqref{HX1}
and \eqref{HV1}, where $\phi_i$ are defined by von Liebig's law
\eqref{Liebm}. Then the number of species such that
$\liminf_{t \to +\infty}  x_i(t) >0$ is not less than
$|{\mathcal B}_M  |$.

\item[{\bf b})] In the model with extinctions, there holds $N_e=0$
    for some initial data (even if $|{\mathcal B}_M  |>0$).
\end{itemize}
\end{proposition}

So, this claim shows in particular that  models with and without
extinctions have completely different behaviour. Remarkably, the
obtained biodiversity estimate does not involve $\gamma_i$, while
the proof makes use the fact that $\gamma_i >0$.

\begin{proof} {  Consider {\bf a}). We use \eqref{averx} that gives
\begin{equation} \label{aa5}
\gamma_i \langle x_i \rangle_t = \langle \phi_i(v) \rangle_t -\mu_i
+ \xi_i(t)
\end{equation}
where $\xi_i(t) \to 0$ as $t\to \infty$. Since $\frac{ r_i v} {K_{i}
+ v}  > \frac{r_i }{K_{i}+ S} v,$ we find from \eqref{aa5} that
$$
\gamma_i  \langle x_i \rangle_t \ge  \frac{r_i }{K_{i}+ S} \langle
v \rangle_t -\mu_i + \xi_i(t).
$$
Now we use \eqref{vvv2}  and then  {\bf a})  is obtained  from the
last inequality as $t \to \infty$.

Let us consider {\bf b}). Let $v(0)$ be small enough such that
$$
\phi_i(v(0)) - \mu_i=-2\kappa_i < 0
$$
holds for all $i=1,\ldots , M$.  Note that $v_l(t) < D_l S_l t +
v_l(0)$.
Then there exists $\tau>0$ such that
$$
\phi_i(v(t)) - \mu_i < -\kappa_i < 0 \quad  t \in (0, \tau).
$$
Therefore,
$$
x_i(t) < x_i(0) \exp(-\kappa_i t), \quad  t \in (0, \tau)
$$
Therefore, if all $x_i(0)$ are sufficiently close to
$X_{\mathrm{ext}}^{(i)}$ all the species
extinct.}
\end{proof}

To illustrate the dependence of the biodiversity  the fundamental
parameters, let us consider an example (see also Example~\ref{ex2}
below).  Let $K_i=K$ for all $i=1, \ldots, M$ and $a_i=\mu_i/r_i$.
Then
$\lambda_i =K({1- a_i})^{-1}$. Therefore, according to
Proposition~\ref{prop3} all species $x_j$ with
\begin{equation} \label{MMM66}
a_j \bigl(1 + \frac{S}{K}\bigr) < \min_{i=1,..., M} ({1- a_i})^{-1}
\end{equation}
survive. The set of such species can have the maximal cardinality
$M$.
Indeed, let $a_i=\frac12 +\tilde a_i$, where $0 <\tilde a_i<\frac12$
and let us assume that $S/K < 1$. Then the right hand side of
\eqref{MMM66} will  be less than $2$. Therefore,  the species with
the property
$$
\tilde a_i <  2\biggl(1+ \frac{S}{K}\biggr)^{-1} -\frac12
$$
do certainly survive. Thus, if all $\tilde a_i$ are small enough,
all the species survive.


\subsection{Three  examples } \label{2ex}

We illustrate the obtained results by simple examples relating them
to the (unified) neutral theory. Recall that neutrality means that
at a given trophic level in a food web, species are equivalent in
birth rates, death rates, dispersal rates and speciation rates, when
measured on a per-capita basis. Mathematically this means that all
basic parameters are almost equal, see \cite {Hubb}. In the all example we suppose
that $D_k$ are large.

\begin{example}  \label{ex1}
For large $d=\min_k D_k$ system  (\ref{Starvd})  becomes a fixed
point problem with a contraction operator, therefore  its solution
$v^{\mathrm{eq}}$ is uniquely determined and
 \begin{equation}\label{asym2}
v_k^{\mathrm{eq}}= S_k -D_k^{-1}  \sum_{i=1}^{M}  c_{ki}
\phi_i(S)  \gamma_i^{-1} (\phi_i(S) -\mu_i)_{+}  + O(d^{-2}).
\end{equation}
By Theorem~\ref{Theor1} all solutions $(x,v)$ of the  system
(\ref{HX1}), (\ref{HV1}) with the Cauchy data (\ref{K16a}) have a
limit
\begin{align*}
v_k(t)&\to S_k+O(d^{-1}),\;\,k=1,\ldots,m,\\
x_i(t)&\to\frac{\phi_i(S)-\mu_i}{\gamma_i}+O(d^{-1}).
\end{align*}
So, if $d$ is sufficiently large the system is persistent. Moreover,
the number $N_{\mathrm{max}}$ is equal approximately to the
cardinality
$$
\Big|\{i\,:\,\frac{\phi_i(S)-\mu_i}{\gamma_i}>X_{\mathrm{ext}}^{(i)}\}\Big|.
$$
\end{example}

\begin{example}[The case of a single resource]\label{ex2}
In the case $m=1$  we set
$\phi_i=\frac{r_i v}{K_i + v}$ and suppose that
$$
\mu_i =\mu, \quad  K_i=K, \quad r_i=r, \quad \gamma_i=  \gamma,
\quad  \bar C=N_e^{-1}\sum_{i=1}^{N_e} c_i
$$
and $X_{\mathrm{ext}}^{(i)}=X_{\mathrm{ext}}$. Let us introduce
parameters
\begin{equation} \label{par}
p=\frac{\mu}{r}, \quad  \tilde S=\frac{S}{K}, \quad  \epsilon=
\frac{\gamma X_{\mathrm{ext}}}{r},\quad R=\frac{K D\gamma}{ r^{2}
\bar C^{1}}.
\end{equation}
We consider
$
v^{\mathrm{eq}}=Ku,
$
where $u$ is to be determined. Then  \eqref{Starvd} becomes
\begin{equation}
  R (\tilde S -u)=N_e F(u),
\label{1Res}
     \end{equation}
where
$$
F(u)=\frac{u}{1+u}\Big(\frac{u}{1+u} - p \Big)_{+, \epsilon}.
$$
  Let
\begin{equation}\label{ueps}
 u_{\epsilon}=\frac{p +\epsilon }{1 - p - \epsilon}.
\end{equation}
We seek  a nontrivial solution of \eqref{1Res} such that $u >
u_{\epsilon}$.
Equation    \eqref{1Res}  implies
\begin{equation}\label{Nf1}
N_*(v^{\mathrm{eq}})  \le N_e \le N_*(v^{\mathrm{eq}})+1,
\end{equation}
where
\begin{equation}\label{Nf1a}
N_*(v)=\left [\frac{ R ( S - v)(K+v)}    {Kv(\frac{v}{K+v} - p)_{+,
\epsilon}}\right]
\end{equation}
and $[x]$ is the floor of $x$. Note that the function
$N_*(v^{\mathrm{eq}})$
is decreasing in $v^{\mathrm{eq}}$. Then since $u >
u_{\epsilon}$ ,  relations  \eqref{ueps} and \eqref{Nf1} allows us
to conclude  that  the maximal possible biodiversity
$N_{\mathrm{max}}$ satisfies
 \begin{equation} \label{mxdiv}
  N_*(K u_{\epsilon})<N_{\mathrm{max}}< N_*(K u_{\epsilon}) +1.
\end{equation}
In the Section~\ref{Rstar} we shall see that this maximum of
biodiversity is realizable for some initial data.

\end{example}

\begin{example}[The multi-resource case]
We assume that $m\ge 2$ and that  the ecosystem is in an equilibrium
state defined by \eqref{sur66e}, \eqref{Starvd} and  \eqref{FNeq}.
We also assume that we are in the neutral position, i.e.
\begin{equation}
     K_{ki}=K, \quad  r_{i}=r,  \ \mu_i=\mu,  \ \gamma_i=\gamma, \
     X_{\mathrm{ext}}=X_{\mathrm{ext}}^{(i)}.
\label{exam21}
\end{equation}
Then setting $\phi(z):=\frac{r z}{K+ z}$ we obtain
\begin{equation}
     \phi_{i}(v)=\min_k \phi(v_k)=\phi(w),
\label{exam22}
\end{equation}
where $  w:=\min_{k} v_k=v_{k^*}$. Then the equilibrium abundances
and resources are determined respectively by
\begin{align}
    x_{i}(w)&=\gamma^{-1} ( \phi(w) - \mu)_{+, \gamma
    X_{\mathrm{ext}}}, \quad 1\le i\le N_e,
\label{exam23}\\
    v_{k}^{\mathrm{eq}}(w)&=\gamma^{-1} D_k \sum_{i=1}^{N_e} c_{ki}
    \phi(w)( \phi(w) - \mu)_{+, \gamma X_{\mathrm{ext}}},\quad 1\le
    k\le m.
\label{exam24}
\end{align}
Since $w=v_{k_*}^{\mathrm{eq}}(w)=v_{k^*}$, we obtain
\begin{equation}
   w=\gamma^{-1} D_{k_*} \sum_{i=1}^{N_e} c_{k_* i} \phi(w)( \phi(w)
   - \mu)_{+, \gamma X_{\mathrm{ext}}}.
\label{exam25}
\end{equation}
The system \eqref{exam23}, \eqref{exam24} and \eqref{exam25}
determines an equilibrium state depending on the index  $k_*\in\{1,
2, \ldots , m\}$. For any fixed $k_*=1, 2, \ldots , m$, we solve
\eqref{exam23}, \eqref{exam24} and \eqref{exam25}. If
$v_k^{\mathrm{eq}} < w$ is valid for all $k \ne k_*$ then $k_*$ is
found.

Using change of the variables $w=Ku,$ where $u$ is a new variable,
we obtain from \eqref{exam25} a similar relation \eqref{1Res} as in
Example ~\ref{ex2}, where all parameters but $R$ are defined by
\eqref{par},  and $R=K D_{k_*}\gamma r^{-2} \bar C_m^{-1}$,
where
$$
 \bar C_m=M^{-1}\sum_{i=1}^M c_{k_*, i}.
$$
Then
\begin{equation}\label{Nf2}
N_*(w)  \le N_e \le N_*(w)+1,
\end{equation}
where $N(v)$ is defined by \eqref{Nf1a}.
\end{example}

\section{Asymptotically sharp estimate of biodiversity for systems
with random parameters}
\label{Rstar}

In the remaining part  of the paper we consider ecosystems with random  parameters for large values of $M$. Moreover, we suppose that the system is in an equilibiria state, i.e. oscillating large time regimes are absent.

First let us  suppose that at the initial moment there are $M\gg1$ species with different (possibly random)  parameters and that the resource supplies $S_k$ are limited. We are going to address the following problem: \textit{How many species $N_e$  will survive}?

It is clear that $N_e$ is a priori bounded by $ M$. Below we obtain
an asymptotically sharp estimate of $N_e$.
We shall assume  that $m=1$  and  the ecosystem has  a globally
convergent  equilibrium state,  $v=v^{\mathrm{eq}}$. We also assume
that $\phi_i$ are defined by the von Liebig law  \eqref{Liebm}.
Then, if the $i$-th species is survived for all times $t>0$, we have
\begin{equation}  \label{ass1}
    \frac{r_i v^{\mathrm{eq}}}{K_i + v^{\mathrm{eq}}} - \mu_i >
    \gamma_i X_{\mathrm{ext}}^{(i)}.
\end{equation}
   Consequently,
\begin{equation}  \label{ass2}
    v^{\mathrm{eq}} > \beta_i,    \quad  \beta_i=\frac{ (\mu_i +
    \gamma_i X_{\mathrm{ext}}^{(i)})K_i}{r_i - \mu_i -\gamma_i
    X_{\mathrm{ext}}^{(i)} },
\end{equation}
and
 \begin{equation}  \label{ass3}
    r_i > \mu_i + \gamma_i X_{\mathrm{ext}}^{(i)}.
\end{equation}
Note also that for $\gamma_i=0$ or $X_{\mathrm{ext}}^{(i)}=0$ the
numbers $\beta_i$ coincide with the parameters $\lambda_i$ defined
by \eqref{lambd}. They determine
survival of species in classical resource competition models without
extinction thresholds and without self-limitation \cite{Hsu05}.

We introduce parameters
\begin{equation}  \label{m3}
\eta_i=D^{-1}c_i(\gamma_i X_{\mathrm{ext}}^{(i)} +\mu_i)
X_{\mathrm{ext}}^{(i)}, \quad i=1,\ldots , M
\end{equation}
and describe the choice of  random values of the model parameters.
Under assumptions  \eqref{Liebm},
the main parameters of our model  are  the coefficients $ \mu_i,
\gamma_i$ and
the vectors $K^{(i)}=(K_{i1}, \ldots, K_{im})$,  where $i=1,\ldots ,
M$, and $r=(r_{1}, \ldots, r_{M})$.
As before, we consider $U_M=\big ({\mu}, \mathbf{K}, \mathbf{r},
\mathbf{c}, \Gamma \big) \in \mathbb{R}^{2M(m +2)}$. Note that
$c_i$ are species specific parameters which not necessary to be
included in  analysis, so we suppose that $c_i$ are fixed.

Let $U_M$ be a random vector with a probability density
function $\xi (U_M)$. This means that the values $U_M$ are defined
by random sampling, i.e. the parameters of the  species are  random
independent vectors that are drawn from the cone $\mathbb{R}^{2M(m
+2)}_+$ by the density $\xi$.

Our basic assumption to $\xi$ then can be formulated as follows:

\begin{assumption}\label{asum1}
The probability density function $\xi$ is a continuous function with
a compact support in $\mathbb{R}^n_+$. Moreover, as $N \to \infty$
the parameters $\beta_i$ are distributed on $(0, +\infty)$
with the smooth  probability density $\rho_0(\beta)$ such that
\begin{equation}  \label{ass4}
    \mathrm{supp}\,\, \rho_0(\beta) \subset (\beta_{\min},
    \beta_{\max}).
\end{equation}
The parameters $\eta_i$ are distributed on $(0, +\infty)$
with a continuous  probability density $\rho_1(\eta)$.  The
densities $\rho_0$ and $\rho_1$ are mutually independent.
\end{assumption}

From Assumption~\ref{asum1} it follows that
the function $\xi$ is positive on $S_{\xi}$, where $S_\xi$  is an
open bounded set.  Assumption~\ref{asum1} also yields that the
mortality rates do not approach  zero and resource consumption is
restricted.
It is supposed that initial data $\bar x_i=x_i(0)$ are random
mutually independent with the density distribution
		\begin{equation*}
     \bar x_{i}  \in  {\mathcal  X}(\bar X,
		 \sigma_{X})
\label{survX1}
\end{equation*}
with the mean $\bar X$ and the deviation $\sigma_{X}$.
The random assembly of the species defines an initial state of the
ecosystem for $t=0$.

Let us order $\beta_i$ so that
$$
\beta_{\min} \le  \beta_1 \le \beta_2 \le \ldots    \le \beta_M \le
\beta_{\mathrm{max}}.
$$
Then
 \begin{equation*}  \label{ass5}
    \beta_{N_e} < v^{\mathrm{eq}} < \beta_{N_e+1}.
\end{equation*}
Therefore, we obtain
 \begin{equation}  \label{ass6+}
    \beta_{N_e+1} > S - D^{-1}\sum_{i=1}^{N_e} c_i \frac{r_i
    v_{eq}}{K_i + v^{\mathrm{eq}}} (\frac{r_i v^{\mathrm{eq}}}{K_i +
    v_{eq}} -\mu_i) \gamma_i^{-1},
\end{equation}
\begin{equation}  \label{ass6-}
    \beta_{N_e} < S - D^{-1} \sum_{i=1}^{N_e} c_i \frac{r_i
    v_{eq}}{K_i + v^{\mathrm{eq}}} (\frac{r_i v^{\mathrm{eq}}}{K_i +
    v_{eq}} -\mu_i) \gamma_i^{-1},
\end{equation}
where by Assumption~\ref{asum1} one has  that $\delta_{\beta} =
\beta_{N_e+1}- \beta_{N_e}=O(M^{-1})$ as $M \to +\infty$.
Therefore, the last two inequalities are asymptotically exact.

\begin{lemma} \label{lemma2}
 {Let us define the probability $\mathrm{Pr}_{\beta}$ by
\begin{equation*}  \label{prob}
    \mathrm{Pr}_{\beta}=\Prob \{ \beta_{n} - \beta_{\min} <
    M^{-1/3} \}
\end{equation*}
Then for sufficiently large $M$
\begin{equation*}  \label{prob1}
    \mathrm{Pr}_{\beta} > 1- C_n  \exp(- c_n M^{-1/6}),
\end{equation*}
where $C_n, c_n$ are positive constants.
}
\end{lemma}

\begin{proof}
{Let us consider the probability $\mathrm{Pr}_{k, \delta}$ that a
random sample consisting of $M$ numbers $\beta_i$  containing
exactly $k$ numbers within the
interval $J=[\beta_{\min}, \beta_{\min}+ \delta]$, where $\delta$ is
a small positive number.  The probability that a random $\beta_i$
distrubuted according
to the density $\rho_0$ lies within $J$ is
$$
p_{\delta}= \int_{ \beta_{\min}}^{\beta_{\min}+ \delta} \rho_0(s)
ds.
$$
Since $\rho_0$ is smooth, the probability can be estimated as
$p_{\delta} < c \delta^2$. Let us define
$$
\mathrm{Pr}_{k, \delta}={M \choose k} p_{\delta}^k (1-
p_{\delta})^{M-k}.
$$
Then
\begin{equation*}  \label{prob2}
\mathrm{Pr}_{k, \delta} <  C_k (M p_{\delta})^k \exp( - 0.5M
p_{\delta}),
\end{equation*}
where the constants $C_k>0$  do not depend on $M$.
We set $\delta=M^{-1/3}$.  Then
\begin{equation}  \label{prob3}
\mathrm{Pr}_{k, \delta} <  C_k  M^{-k/6} \exp( - c(k) M^{-1/6}).
\end{equation}
We observe that
\begin{equation}  \label{prob4}
 \mathrm{Pr}_{\beta}=1 - \sum_{k=1}^n \mathrm{Pr}_{k, \delta}> 1 -
 n\max_{k\in \{1,\ldots ., n\}}  \mathrm{Pr}_{k, \delta}.
\end{equation}
For fixed $n$ and $M$ large enough one has
$$
\max_{k\in \{1,\ldots ., n\}}  \mathrm{Pr}_{k,
\delta}=\mathrm{Pr}_{n, \delta}.
$$
Combining this with \eqref{prob3} and \eqref{prob4} yields that for
sufficiently large $M$
$$
\mathrm{Pr}_{\beta}> 1 - n \mathrm{Pr}_{n, \delta} > 1 -  C_n
(M)^{-n/6} \exp( - c_n M^{-1/6}).
$$
The latter estimate proves \eqref{prob1}}.
\end{proof}

Proposition  \ref{prop2} shows that $N_e \ll  M$ and $\rho_0$ is a
smooth density. Therefore, by Lemma~\ref{lemma2} we conclude that
$$
\beta_{N_e} - \beta_{\min} <  M^{-1/3}
$$
and
$$
(\frac{r_i v_{eq}}{K_i + v_{eq}} -\mu_i)  = X_{\mathrm{ext}}^{(i)} +
O(M^{-1/3}),
$$
with a probability exponentially close to $1$ for large $M$. This
yields the following relation which is asymptotically sharp:
\begin{equation*}  \label{ass7}
    \beta_{\min}  + \sum_{i=1}^{N_e} \theta_i=S + O(M^{-1/3}).
\end{equation*}
Since the densities $\rho_0$ and $\rho_1$
are mutually independent, the sum in the left hand side of the last
equation can be replaced by $N_e \langle \theta \rangle$, where
$\langle \theta \rangle = E \theta$ is the averaged value of
$\theta$. Finally one has
  \begin{equation}  \label{assF}
    {N_e} =(S - \beta_{\min})  \langle \theta \rangle^{-1} +
    O(M^{-1/3}).
\end{equation}

The latter estimate considerably refines our previous result
\cite{KVU16}.  Moreover,  we conclude that the final state of
ecosystem originally consisting of many
species with random parameters can be described by the so-called
$R^*$ rule. The essential parameters $\beta_i$ are well localized in
a narrow domain
that follows from \eqref{ass6+} and \eqref{ass6-}
(see Fig. \ref{Fig2}) that confirms the $R^*$ rule. Recall that the
$R^*$ rule (also called the \textit{resource-ratio hypothesis}) is a
hypothesis in community ecology that attempts to predict which
species will become dominant as the result of competition for
resources. It predicts that if multiple species are competing for a
single limiting resource, then species, which survive at the lowest
equilibrium resource level,  outcompete all other species
\cite{Til2}. Note that  a generalized resource ratio value $\beta_i$
for small $\gamma_i X_{\mathrm{ext}}^{(i)}$ is
$$
\beta_i \approx \lambda_i,
$$
where $\lambda_i$ are defined by \eqref{lambd},
   and it involves the mortality,  sharpness of consuming rates, and
   consuming rates.
Note, however, that the proof of  \eqref{assF} cannot be extended
for $m>1$.

To conclude this section, let us remark that \eqref{assF}   improves
\eqref{averv61}. Indeed,   $N_e$ from \eqref{assF} satisfies
\eqref{averv61}. Moreover,
a  difference between those estimate is that in \eqref{assF} the
nominator $S$ is replaced by $S - \beta_{\min}$, therefore, for large
resource supply  these estimates
coincide.

\begin{figure}[t]
\includegraphics[width=0.5\linewidth]{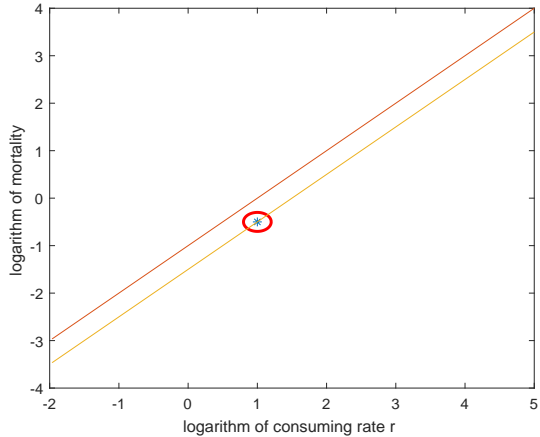}
\caption{This plot shows the case when an ecosystem in a neutral
state where all parameters are localized.
 We assume that the population parameters are subject to log-normal
 laws. The corresponding averages $E \ln r=1$ and $E \ln \mu=-0.5$.
 These parameters
are shown by  the star point in the center of a circle of the small
radius $0.5$.
All the species parameters are localized inside this circle.
The half-plane $\Omega_S$ consists of the points that lie below the
upper right straight line $L_S$ defined by eq. $y=x-S$.  As it follows from our analytical arguments,
if species parameters lie in $\Omega_S$ then that species survives; otherwise it would be extinct.  For $S=1$ all species survive. If resource supply $S$ slightly diminishes, the  line  $L_S$ shifts to the bottom and only the half species rest inside the small circle
the second half of the species disappears.}
\label{Fig2}
\end{figure}

\section{Extinctions and mass extinctions } \label{extmass}

\subsection{Mass extinctions }
Let us revisit our Examples~\ref{ex1} and \ref{ex2} above, which
describe  the localized case (Assumption~\ref{asum1}), where
parameters of all species are localized in a narrow domain that consistent with the $R^*$-rule. Note
that the nontrivial solution $u \ne 0$ of eq. \eqref{1Res} exists
under condition $u_{\epsilon} <
\tilde S$  thus
 \begin{equation*} \label{mxext}
  \frac{\mu +\gamma X_{\mathrm{ext}} }{ r - \mu   - \gamma
  X_{\mathrm{ext}}}  < S/K.
\end{equation*}

The violation of this condition leads to a mass extinction, when all
the species disappear.
In \cite{KVU16}, the following localization effect is
described: in a long time evolving species population, which was
randomly assembled at the
initial moment  (see Assumption~\ref{asum1} above),  all species
parameters converge to a limit value.  Then one can show that this
localized population state has an interesting property:

{\bf All species die together: } {\em  As a result of a long
evolution, initially randomly assembled  population approaches to a
neutral fitness invariant state, where, if an extinction occurs
then the extinction is a mass extinction, i.e. all species die
simultaneously.}

So, we conclude that
\begin{enumerate}

 \item extinctions are stronger if the variations of species
     parameters are smaller, i.e. localization of parameters is
     higher;

\item for well localized parameters ($\sigma_{\mu}$ and $\sigma_r$
    are small)  there is  a critical domain of parameter values
    when a small variation in the resource supply
$S_k$  can lead to   a mass extinction.

\end{enumerate}

These properties can be illustrated  by Fig. \ref{Fig2}.


We assume $m=1$ and  that coefficients $c_i, r_i, \gamma_i$ and
$\mu_i$ satisfy
\begin{align} \label{assum100}
C_- r<  &r_i  <C_{+} r, \quad C_- c <  c_i  <C_{+} c,  \quad 1 \le i
\le M,\\
\label{assum101}
C_- \gamma < &\gamma_i  <C_{+} \gamma, \quad C_- \mu <  \mu_i  <C_{+}
\mu, \quad 1 \le i \le M
\end{align}
where $a,  c, \gamma, r$ are characteristic values of the
corresponding coefficients, $C_{\pm}$ are positive constants
independent of
$M,r,c, \gamma, \mu$.
Let us introduce the stress parameter by

\begin{equation*} \label{Stress}
P_{\mathrm{stress}}= \frac{c r^2}{\gamma D S}.
\end{equation*}

Let us find a relation between that parameter and the biodiversity
robustness, which we define as
a coefficient $R_b$ in the relation
\begin{equation} \label{Rob}
\frac{\Delta N_e}{N_e}= R_b^{-1}\frac{ \Delta S}{S},
\end{equation}
where $\Delta S < 0$ is a small variation of the resource supply
$S$ caused for instance by a climate change, and
$\Delta N_e \le 0$ is the corresponding variation of the number of
coexisting species $N_e$.
We suppose
that relation  \eqref{assF}  is fulfilled. Then comparing
\eqref{assF} and \eqref{Rob} we see
that
$$
 \langle \theta \rangle =M^{-1} \sum_{i=1}^M \theta_i,
$$
where $\theta_i$ are defined by \eqref{m3}.  By  \eqref{m3} we note
that in the case
of large biodiversity $N_e \gg 1$ the coefficient $ \langle \theta
\rangle$ of
same order that $ SP_{\mathrm{stress}}$. In fact, if all $r_,
\gamma_i$ and $\mu_i$ of the same order (as it was supposed above,
in \eqref{assum100} and \eqref{assum101}), then $X_{\mathrm{ext}}$
has order $r/ \gamma$. Substituting this
relation in  \eqref{m3}, we obtain that $R_b$ has the same order $
SP_{\mathrm{stress}}$. Note that $P_{\mathrm{stress}}$
is a dimensionless parameter.

\section{Conclusions} \label{Concl}

The human activity affects ecosystems restricting   their resources
and producing climatic changes.  The key question is about
consequences of that activity, it will be an abrupt
change of biodiversity comparable with famous mass extinctions such
as the end-Permian extinction (252 million years ago (Ma)), and
even more severe the Phanerozoic (the past 542 Ma), or we will
observe a relatively smooth decline of biodiversity.

The history of the Earth system show that some changes were gradual,
but others were   sharp and produced catastrophic mass extinctions.
To understand that wait us in Future
we should investigate ecosystem biodiversity, find key factors
affecting this biodiversity and estimate the sensitivity of
ecosystems with respect to environmental parameters.

A number of  works are devoted to biodiversity problem beginning
with \cite{Volterra} (see, for example,  \cite{Hardin, Alles2,
Hu61, Til2, HuWe99, HuWe2001} this list does not pretend to be
complete).
However, mathematical methods developed up to now, does not allow us
to obtain explicit estimate of ecosystem biodiversity via the
fundamental and experimentally measurable ecosystem parameters.
The main difficulty is competitive exclusion principle. According to
\cite{Hsu05} for the case of one fixed resource supply we obtained
two situations. The species survival depends on ''break-even"
parameters
$\lambda_i$ (see also  \eqref{lambd} below). If $\lambda_1$ is less
than all the rest $\lambda_i$ only the species number $1$ survives.
To obtain coexistence of $m$
species observed in many natural resource competition ecosystems
(\cite{Hu61}) we must set $\lambda_1=\lambda_2=\ldots=\lambda_m$,
where we, however, can take $m$ in an arbitrary way (only to be
less than general species number).
To get around this difficulty a number of different approach were
suggested, for example, dynamical chaos in systems with more than
$2$ resources \cite{HuWe99, HuWe2001}, to take into account temporal
variations of
parameters etc. (see  for example \cite{Volm} for an overview) but
these ideas does not permit us  to obtain an explicit analytic
estimates for the number of coexisting species.

In this paper, we develop further two relatively new ideas: namely we take into account weak self-limitation effects \cite{Roy2007} and extinction thresholds \cite{KVU16}.  Although the modified models seem to be more
sophisticated, these effects allow us to  regularize the problem mathematically even in the case of a single  resource and also to obtain
explicit biodiversity estimates for diverse scenarios.

We also extend results \cite{Rothman} and show that
for plant or plankton systems,  the magnitude
of the future catastrophe depends not only on the level of climate
variations.
The structure of ecosystems also influences the number of species
that go to extinction. We establish estimates of biodiversity via
the system parameters such as the mortality rate, the consuming rate, teh sharpness of consuming rate, self-limitation  and the critical species abundance. Furthermore, we show that the ecosystems where these parameters are well localized are less stable that the ecosystems with  large variations in these parameters.  We also investigate how biodiversity depends on structure climate-ecosystem interactions.

In summary, or  main results are:

\begin{enumerate}

\item The dynamics of model \textit{without} extinctions is determined for large resource turnovers: we show that all trajectories converge to a unique  equilibrium, system has no memory and   forgets initial data completely.

\item For large resource turnovers, the  dynamics of the model \textit{with} extinctions is stable in a weaker sense: any trajectory goes to an equilibria which now may depend on the initial data. Thus, the ecosystem has memory  and the final equilibria depend on initial data.

\item For  models with extinctions, the ecosystem biodiversity can be
    explicitly estimated by a  relation involving only measurable parameters (turnover and mortality rates, resource supplies and the species content coefficients). This estimate is  asymptotically sharp under certain natural assumptions and in the case when the neutral theory correctly describes the ecosystem state. We show also that such a situation naturally occurs when the ecosystem is under stress, i.e. originally there exists species with randomly distributed  parameters and the most of them disappear as a result of resource limitations.

\item In the stress situation, we also show that the final state of the system (after extinctions) can be determined according to the $R^*$ rule.

\item The ecosystem robustness is investigated and conditions for mass
    extinctions are obtained. The robustness depends, in particular,  on the dimensionless stress parameter introduced in \cite{KVU16}. Systems with a large stress parameter is less sensitive with respect to resource decrease, however, if  an extinction is happened. Such an event is catastrophic and destroys  an essential  part of that ecosystem.

\end{enumerate}


\section*{Acknowledgments}

The authors would like to thank the reviewers for their detailed comments and suggestions for the manuscript.  We are also grateful to L.~
Mander for a question on extinction problem that have stimulated the
present research. The second author was supported by Link\"oping University, by  Government of Russian Federation, Grant 08-08
and via the RFBR grant  16-01-00648.


\bibliographystyle{plain}

\def\cprime{$'$}

\end{document}